\documentclass[11pt]{amsart}
\usepackage{pstricks}
\usepackage{color}
\theoremstyle{plain}
\newtheorem{thm}{Theorem}[section]
\newtheorem{theorem}[thm]{Theorem}

\newtheorem{lemma}[thm]{Lemma}

\newtheorem{proposition}[thm]{Proposition}
\theoremstyle{definition}

\newtheorem{definition}[thm]{Definition}

\numberwithin{equation}{section}

\newcommand{\Pic}{{\rm Pic}}

\newcommand{\Spec}{{\rm Spec \,}}

\newcommand{\sO}{{\mathcal O}}

\newcommand{\sU}{{\mathcal U}}

\newcommand{\A}{{\mathbb A}}

\newcommand{\C}{{\mathbb C}}

\newcommand{\F}{{\mathbb F}}

\renewcommand{\P}{{\mathbb P}}
\newcommand{\Q}{{\mathbb Q}}
\newcommand{\R}{{\mathbb R}}

\newcommand{\Z}{{\mathbb Z}}

\title [Salem numbers of maximal degree]{Automorphisms of elliptic  K3 surfaces and Salem numbers of maximal degree}
\author{H\'el\`ene Esnault, Keiji Oguiso and Xun Yu} 
\address{Freie Universit\"at Berlin, Arnimallee 3, 14195, Berlin,  Germany}
\email{esnault@math.fu-berlin.de}
\address{Department of Mathematics, Osaka University, Toyonaka 560-0043, Osaka, Japan and Korea Institute for Advanced Study, Hoegiro 87, Seoul, 
133-722, Korea}
\email{oguiso@math.sci.osaka-u.ac.jp}
\address{Center for Geometry and its Applications, POSTECH, Pohang 790-784, Korea}
\email{yxn100135@postech.ac.kr}
\thanks{The first  author is supported by  the Einstein program and the ERC
Advanced
Grant 226257. The second author is supported by JSPS Grant-in-Aid (S) No 25220701, JSPS Grant-in-Aid (S) No 22224001, JSPS Grant-in-Aid (B) No 22340009, and by KIAS Scholar Program.}
 \dedicatory{Dedicated to Professor Shing-Tung Yau on the occasion of his
sixty-fifth birthday.}

\date{November 17, 2014}
\begin{document}
\begin{abstract}  Using  elliptic structures, we show that any supersingular K3 surface of Artin invariant $1$ in characteristic $p \not= 5$, $7$, $13$ has
an automorphism the  entropy of which  is the natural logarithm of a Salem number of degree $22$.

\end{abstract}
\maketitle
\section{Introduction}\label{intro}

If $p$ is a prime number,  there is an Artin invariant $1$ supersingular K3 surface $X(p)$, defined over the congruence field $\F_p$. It is unique up to isomorphism (see  Section~\ref{sec:notations}).   Since Ogus' seminal work \cite{Ogu79}, \cite{Ogu83}, this surface has been studied from various viewpoints.  

\medskip

In \cite{ES13} it is shown that, as in characteristic $0$,  in positive characteristic  the maximum  of the absolute values of the algebraic integers, which are  eigenvalues of an  automorphism of a smooth projective surface acting on its $\ell$-adic cohomology, is taken on the N\'eron-Severi group.  Thus by analogy with complex geometry, one calls entropy the logarithm of this maximum. One knows that  the entropy of an automorphism of a K3 surface is either $0$ or the logarithm of a Salem number (see  Section~\ref{ss:salem}), then of degree at most the rank of the N\'eron-Severi group, thus at most $20$ for projective K3 surfaces in characteristic $0$.

\medskip

Over  $k=\bar \F_p$, 
Jang  (\cite{Jan14})  showed that the image of the canonical representation
$${\rm Aut}\,(X(p)) \to {\rm GL}\, (H^0(X(p), \omega_{X(p)})) \simeq k^{\times}$$
is isomorphic to the cyclic group  of order $p+1$. 
In particular, for $p$ large,  there are automorphims which are not geometrically liftable to characteristic $0$.
His proof relies on Ogus' Torelli theorem \cite{Ogu83}. From this, and from Shioda's study  of Mordell-Weil lattices \cite{Sh90},  it is deduced in \cite{EO14} that for $p$ very large,  there are automorphisms of $X(p)$  of positive entropy which  are not geometrically liftable to characteristic $0$ (see \ref{ss:liftingautomorphism} for the definition of this liftability notion).

\medskip

The main result of this note is
\begin{theorem} \label{IntroThm1}
Let $p \not= 5$, $7$, $13$. Then there is an automorphism $f \in {\rm Aut}\, (X(p))$, defined over $\bar \F_p$,  the entropy  of which is the logarithm of a Salem number of degree $22$. 
\end{theorem}
In particular, those automorphisms are not geometrically liftable to characteristic $0$. 
The result is known over $k=\bar \F_p$ for $p = 2$ (\cite{BC13})  and $p=3$  (\cite{EO14}). 
\medskip

While for $p=2$ the proof relies on the very detailed study of $X(2)$ in \cite{DK02}, and for $p=3$,  it is computer aided and relies on the explicit study of ${\rm Aut}(X(3))$ in \cite{KS12}, our proof  of Theorem~\ref{IntroThm1} is abstract and   based on the theory of the Mordell-Weil groups of $X(p)$ by Shioda \cite{Sh13}. 
We show the following strengthening of Theorem~\ref{IntroThm1}:
\begin{theorem} \label{IntroThm2}
Let  $k$ be an algebraically  closed field of characteristic $p \ge 0$. Let $X$ be a K3 surface over $k$ of Picard number $\rho =2d \ge 4$. Assume that $X$ has two non-isomorphic elliptic fibrations $\varphi_1 : X \to  \P^1$ $(i=1$, $2$), such that the Mordell-Weil group ${\rm MW}\, (\varphi_1)$ of $\varphi_1$ is of maximal rank, that is $2d - 2 $, and ${\rm MW}\, (\varphi_2)$ is of positive rank. Then $X$ has an automorphism $f$, the entropy of which  is the  logarithm of a Salem number of degree $2d$.
\end{theorem}

One deduces Theorem \ref{IntroThm1} from Theorem \ref{IntroThm2} using  Shioda's theorem 
\cite{Sh13}:

\begin{theorem} \label{IntroShioda}
Let $p = 11$ or $p > 13$. Then $X(p)$ admits an elliptic fibration of  Mordell-Weil rank  $20 = 22 - 2$.
\end{theorem}

This also explains the restriction  on $p$ in Theorem \ref{IntroThm1}. One may think that this restriction should not appear in Theorem~\ref{IntroThm1}.  More generally, it is likey that 
Theorem~\ref{IntroThm2},  Theorem~\ref{thm:BlancCantat} (mimicing \cite{BC13}) and 
Theorem~\ref{thm:StableSublattice} 
could have a larger range of applications. 
\medskip

We do not address in this note some questions of more arithmetical flavor. We know that the N\'eron-Severi group of $X(p)$ is defined over $\F_{p^2}$ (\cite{Sch12}). Over which field are those automorphisms of Theorem~\ref{IntroThm1} defined? This depends on the field of definition of the Mordell-Weil groups in Theorem~\ref{IntroShioda}.  Further we know that Salem numbers of bounded degree  are discrete. This raises the question  whether or not the minimal Salem number  of degree $22$  arises as the logarithm of the entropy of an automorphism on a supersingular K3 surface (see \cite{Mc11} and references there). 
We also know that powers of Salem numbers are Salem numbers. 
Given  a Salem number of degree $22$ which is the power of another Salem number of degree $22$,    and such that its logarithm   is the entropy of  an automorphism  $f$ on a supersingular K3 surface, when is  $f$ itself   the power of an automorphism?
Finally
it would be interesting to relate this work to  \cite{GMc02}, \cite{Mc02}  in which the authors show that any unramified degree 22 Salem number  is the logarithm of the entropy of an automorphism of a non-projective complex K3 surface. 
\\[1cm]

\noindent
{\bf Acknowledgement.}  \\[.2cm]
We thank  S. Cantat, B. Gross, and C. T. McMullen for discussions. The first author thanks the department of mathematics of Harvard University for hospitality, the third author thanks  T. Hibi for financial support during his stay at Osaka University.

\section{Preliminaries on K3 surfaces and liftings} \label{sec:notations}
\noindent 

In this section, we fix notations and recall basic facts on K3 surfaces and liftings from \cite{EO14} and references therein.

\subsection{K3 surfaces} \label{ss:k3}

Let $X$ be a K3 surface defined over an algebraically closed field $k$ of characteristic $p \ge 0$, that is, $X$ is a smooth projective surface defined over $k$ such that $H^1(X, \sO_X) = 0$ and the dualizing sheaf is trivial : $\omega_X \simeq \sO_X$. We denote by ${\rm NS}(X)$ the N\'eron-Severi group of $X$. Then the Picard group
${\rm Pic}\, (X) $ is isomorphic to  N\'eron-Severi group  ${\rm NS}\, (X)$, which is a free $\Z$-module of finite rank. The rank of ${\rm NS}\, (X)$ is called the Picard number of $X$ and is denoted by $\rho(X)$.  It is at least $1$ as $X$ is assumed to be projective, and at most $22$, the second $\ell$-adic Betti number. In characteristic $0$, it is at most $20$ by Hodge theory.
The intersection form $(*, **)$ on ${\rm NS}\, (X)$ is of signature $(1, \rho(X) -1)$ and $({\rm NS}\, (X), (*,**))$ is then an even  hyperbolic lattice. The dual $\Z$-module ${\rm NS}\, (X)^* := {\rm Hom}_{\Z}({\rm NS}\, (X), \Z)$ is regarded as a $\Z$-submodule of ${\rm NS}\, (X) \otimes \Q$, containing ${\rm NS}\, (X)$ through the intersection form $(*, **)$ which is non-degenerate. The quotient module ${\rm NS}\, (X)^*/{\rm NS}\, (X)$ is called the discriminant group of $X$.

The surface $X$ is called supersingular  if   $\rho(X) = 22$, the maximum possible value. 
(As the Tate conjecture is not yet  proven  for $p=2$, one should rather say  Shioda  supersingular in this case, but we 
consider only supersingular K3 surfaces in odd characteristic  in this note).

 Artin \cite{Ar74}  proved that the discriminant group ${\rm NS}\,(X)^*/{\rm NS}(X)$ of a supersingular K3 surface $X$ is $p$-elementary, more precisely,  as an abelian group
$${\rm NS}\,(X)^*/{\rm NS}\,(X) \simeq (\Z/p)^{2\sigma(X)}$$
where $\sigma(X)$ is an integer such that $1 \le \sigma(X) \le 10$. The integer $\sigma(X)$ is called the Artin invariant of $X$. Let $\sigma$ be an integer such that $1 \le \sigma \le 10$. Then the supersingular K3 surfaces over $k$ with   $\sigma(X) \le \sigma$   form $(\sigma -1)$-dimensional family over $k$. 
There is, up to isomorphism, a unique Artin invariant $1$ supersingular K3 surface $X(p)$  in each positive characteristic $p>0$. It is defined over $\F_p$ and its N\'eron-Severi group is defined over $\F_{p^2}$, not over $\F_p$ (\cite{Ogu79}, \cite{Sch12}). 
  In many senses, $X(p)$ are the most special K3 surfaces. The uniqueness of $X(p)$ in particular shows that $X(p) \simeq {\rm Km}\, (E \times_{\F_p} E)$ for any supersingular elliptic curve $E$ over $\F_p$ (\cite{Ogu79}, \cite{Sh75}). 

\subsection{Lifting of K3 surfaces}  \label{ss:lifting}
 Let $X$ be a K3 surface defined over  a perfect field $k$ of positive characteristic $p > 0$, and $R$ be a discrete valuation ring with residue field $k$ and field  of fractions  $K = {\rm Frac}\, (R)$ of characteristic $0$. We call a proper flat  (thus smooth) morphism  of schemes $X_R\to \Spec R$, which restricts to $X\to {\rm Spec}(k)$,
 a characteristic $0$ model of $X/k$.  The generic fiber $X_K= X_R \otimes_R K$  is a  K3 surface. For any field $L$ containing $K$, the K3 surface $X_L=X_K\otimes_K L$ is called a lift of $X$ over $L$. Characteristic $0$ lifts of $X$ are lifts over some $L$ as above. 
 They  are unobstructed (\cite{Del81}).
\subsection{Geometric lift of an automorphism of a K3 surface}(See \cite[Section~2.4]{EO14}).  \label{ss:liftingautomorphism}
Let $X$ be a K3 surface defined over a perfect field  $k$ of positive characteristic $p > 0$ 
and $X_R\to \Spec R$ be a characteristic $0$ model.
 Recall  (\cite[X,~App.]{SGA6}) that  one has a {\it specialization homomorphism }
 $sp: \Pic(X_{\bar K})\to \Pic(X)$ {\it on the Picard group},  which is defined by spreading out and restriction. It is injective as recognized in $\ell$-adic cohomology, on which the specializaion is an isomorphism (\cite[V,~Thm.~3.1]{SGA4.5}).

 One has a {\it restriction homomorphism  } ${\rm Aut}(X_R/R)\to {\rm Aut}(X)$. 
 One defines  the subgroup ${\rm Aut}^e (X_{\bar K}/{\bar K})\subset {\rm Aut}(X_{\bar K}/{\bar K})$ consisting of those automorphisms which lift to some model $X_R\to {\rm Spec} \ R$. (Here {\it e}  stands  for extendable).   The group law is defined by base change and the composition of automorphisms. Then the restriction homomorphism yields a {\it specialization homomorphism} 
$ \iota: {\rm Aut}^e(X_{\bar K}/\bar K)\to {\rm Aut}(X/k) .$
   Moreover, $sp$ is equivariant under $\iota$.
   In addition,  as automorphisms are recognized on the associated formal scheme, and $H^0(X, T_{X/k})=0$, the specialization homomorphism 
$\iota$ is injective (see \cite[~Lem.~2.3]{LM11}) .
 
 An automorphism   $f$ in ${\rm Aut}(X)$   {\it geometrically liftable to characteristic } $0$ if 
it is in the image of the specialization homomorphism $\iota$ for some model $X_R/R$.

\begin{theorem}(See \cite[Proof~of~Thm.6.4]{EO14}) \label{NoLiftPosEntropy} 
Let $X$ be a supersingular K3 surface and $f \in {\rm Aut}\, (X)$. Assume that 
the characteristic polynomial $f^*|{\rm NS}\, (X)$ is irreducible in $\Z[t]$. 
Then $f$ is never geometrically liftable to characteristic zero.
\end{theorem}

\begin{proof} If $f$ lifted  geometrically  to characteristic $0$, say to $g \in {\rm Aut}\, (X_{\bar K})$ under $X_R \to \Spec R$, then, as explained above, the specialization map $\iota: {\rm NS}(X_{\bar K})\hookrightarrow {\rm NS}(X)$ would be equivariant with respect to $g^*$ and $f^*$. In particular, the minimal monic polynomial $m_g(t) \in \Z[t]$ of $g^* | {\rm NS}(X_{\bar K})$, which has degree $\le 20$, would divide the minimal monic polynomial $m_f(t) \in \Z[t]$ of $f^* | {\rm NS}(X)$ in $\Z[t]$, which is irreducible and of degree $22$ by assumption, a contradiction.

\end{proof}

\section{Preliminaries on Salem numbers and entropy}  \label{ss:salem}
\label{sec:notations2}
\noindent 
In this section,   we
  recall the definition of entropy and Salem numbers, again from \cite{EO14} and  the references therein. 

\medskip

In what follows, $L = (\Z^{1+t} , (*,**))$  is a hyperbolic lattice, i.e., a pair  consisting of a free $\Z$-module of rank $1+t$ and a $\Z$-valued symmetric bilinear form $( \ , \ )$ of $\Z^{1+t}$ of signature 
$(1, t)$ with $t > 0$. For any ring $K$, one denotes the scalar extension of $L$ to $K$ by $L_K$.

We denote by ${\rm O}(L)$  the orthogonal group  of the quadratic lattice $L$. It is an algebraic group defined over $\Z$. 
  The determinant 
$$\det : {\rm O}(L) \to \{\pm 1\}\,\, $$
is a surjective homomorphism of  algebraic groups. Its kernel ${\rm SO}(L) \subset {\rm O}(L)$ 
is a closed index $2$ algebraic subgroup,  
 and is the  identity component  of ${\rm O}(L)$.  As an algebraic group, ${\rm SO}(L)$ is geometrically connected.
 Moreover, ${\rm SO}(L)(\R)$ is a connected  real Lie group, has index $2$ in
 the real Lie group ${\rm O}(L)(\R)$, thus is its identity component in the real topology.


The subspace 
$$P:=\{x \in L_{\R} |  \ (x^2) > 0\} \subset L_{\R}$$
consists of  two connected components $\pm C$ 
in the real topology. 
By continuity,  for $g \in {\rm O}(L)(\R)$, one has $g(C) = C$ or $g(C) = -C$, where the second case occurs.
Thus
$${\rm O}^{+}(L_{\R}) := \{ g \in {\rm O}(L)(\R) | \ g(C) = C \}
\,\,$$
is an index $2$  closed subgroup of the real Lie group  ${\rm O}(L)(\R)$. Thus 
${\rm O}^{+}(L_{\R})$ 
is  a  real Lie subgroup of ${\rm O}(L)(\R)$, which is not the subgroup of $\R$-points 
 of  an algebraic subgroup of ${\rm O}(L)$.

One defines 
$${\rm O}^{+}(L) := {\rm O}(L) (\Z) \cap {\rm O}^+(L_{\R})\,\, ,$$
$${\rm SO}^{+}(L) := {\rm O}(L) (\Z) \cap {\rm O}^+(L_{\R}) \cap {\rm SO}(L)(\R)\,\, .$$
The groups ${\rm O}^{+}(L)$ and ${\rm SO}^{+}(L)$ are abstract subgroups of ${\rm O}(L)(\Z)$ 
of index at most $4$.

\medskip

For application for K3 surfaces, we take $L$ to be the N\'eron-Severi lattice $ {\rm NS}\, (X)$ and $C$ to be  the connected component of $P$ containing the ample cone. Thus  ${\rm Aut}\,(X)^*$, the representation of ${\rm Aut}\, (X)$ on ${\rm NS}\, (X),$ lies in ${\rm O}^+({\rm NS}(X))$. Moreover,  ${\rm Aut}\, (X)^{*0} := {\rm Aut}\,(X)^* \cap {\rm SO}^{+}({\rm NS}\,(X))$  is a subgroup of ${\rm Aut}^*(X)$ of index at most $2$. 
\medskip

We call a polynomial $P(x) \in \Z [x]$ a {\it Salem polynomial} if it is irreducible, monic,  of even degree $2d \ge 2$ and the complex zeroes of $P(x)$ are of the form ($1 \le i \le d-1$):
$$a > 1\,\, ,\,\, 0 < a < 1\,\, ,\,\, \alpha_i, \overline{\alpha}_i \in S^1 := 
\{z \in \C\, \vert\, \vert z \vert = 1\} \setminus \{\pm 1\}\,\, .$$

\begin{proposition}\label{salem} Let $f \in {\rm O}^+(L)$. Then, the characteristic polynomial of $f$ is the product of cyclotomic polynomials and at most one Salem polynomial counted with multiplicities. 
\end{proposition}    
\begin{proof} As mentioned in \cite[Prop.~3.1]{EO14}, this is well-known. See 
\cite{Mc02}, \cite{Og10}.
\end{proof} 
\begin{definition}\label{entropy}  \begin{itemize}
\item[i)]
 For $f$ as in Proposition~\ref{salem}, we define the entropy 
$h(f)$ 
of $f$ by
$$h(f) = \log ( r(f) ) \ge 0\,\, ,$$
where $r(f)$ is the spectral radius of $f$, that is the maximum of the absolute values of the  complex
 eigenvalues  of $f$ acting on $L$. 
\item[ii)]
For a smooth projective surface $S$ and an automorphism $f$ on it,  one defines the entropy   $h(f)$  by
$$h(f) = \log r(f^* | {\rm NS}\, (S)) $$ 
where $f^*$ is the action on ${\rm NS}(S)$ induced by $f$. 
 \end{itemize}
\end{definition}
This definition is consistent to the topological entropy of automorphisms of smooth complex projective surfaces (\cite{ES13}).
\section{Two observations from group theory} 
\label{sec:grouptheory}
\noindent
In this section, we shall prove  Theorem~\ref{thm:BlancCantat}, relying on \cite{BH04} and \cite{BC13}, and   Theorem~\ref{thm:StableSublattice}, relying on \cite{Og09}. They are  crucial for our main theorems \ref{IntroThm2} and \ref{IntroThm1}.  
\begin{thm} \label{thm:BlancCantat}
Let $L$ be a hyperbolic lattice of even  rank $2d$ and $G \subset {\rm SO}^{+}(L)$ be  a subgroup. Assume that $G$ has no $G$-stable $\R$-linear subspace of $L_{\R}$ other than $\{0\}$ and $L_{\R}$. Then there is  an element $g \in G$, the  characteristic polynomial of which  is a Salem polynomial of degree $2d$. 
\end{thm} 

 The proof below  mimics \cite[p.~15]{BC13}, where the authors handle the case of $L={\rm NS}(X(2))$,  and deduces the statement from \cite[Prop.~1]{BH04}. 
We slightly clarify  their argument to make it fit with  Theorem~\ref{thm:BlancCantat}.

\begin{proof} 
Recall  \cite[Prop.~1]{BH04}, in which neither the evenness of the rank $L$ nor $G \subset {\rm SO}^{+}(L)$ are necessary assumptions:

\begin{thm} \label{thm:delaHarpe}
Let $G \subset {\rm O}(L)(\R)$ be an abstract  subgroup. Assume that $G$ has no $G$-stable $\R$-linear subspace of $L_{\R}$ other than $\{0\}$ and $L_{\R}$. Then the Zariski closure of $G$ in ${\rm O}(L_{\R})$ contains  ${\rm SO}(L_{\R})$.  In particular, if  in addition  $G\subset {\rm SO}(L)(\R)$, then its Zariksi closure in ${\rm SO}(L_{\R})$ is 
${\rm SO}(L_{\R})$.
\end{thm}

As is well known, if $L$ is any non-degenerate quadratic lattice, and $g \in {\rm SO}(L)(\Z)$, then if the rank of $L$ is odd, $1$ is an eigenvalue of $g$.  Indeed, if $Q, M$ are the matrices of the form and $g$ in a chosen basis, then 
$^tMQ(M-{\rm Id})=({\rm Id} -^tM)Q$ thus ${\rm det}(M-{\rm Id})=-{\rm det}(M-{\rm Id}) \in \Z$. We use now that in Theorem~\ref{thm:BlancCantat}, $L$ is even:

\begin{proposition} \label{prop:evenness}
There is an element $g \in {\rm SO}(L)(\R)$ such that no eigenvalue of $g$ is a root of unity. 
\end{proposition}

\begin{proof}
We may choose a real basis 
$$\langle v_1, v_2, v_3, \cdots, v_{2d} \rangle$$ 
of $L_{\R}$ under which the bilinear form $(*,**)$ is represented by the matrix $Q := (1, -1, -1, \cdots, -1)$. We identify $\R$-linear maps and $2d \times 2d$-matrices with real entries via this basis. 
Consider the $\R$-linear map of $L_{\R}$ given by the matrix 
$$M = P \oplus R_2 \oplus \cdots \oplus R_d$$
where 
$$P= \left(\begin{array}{rr}
\sqrt{2} & 1\\
1 & {\sqrt{2}} \end{array} \right)\,\, ,\,\, 
R_i = \left(\begin{array}{rr}
\cos 2\pi\sqrt{2} & -\sin 2\pi\sqrt{2}\\
\sin 2\pi\sqrt{2} & \cos 2\pi\sqrt{2} 
\end{array} \right)\,\, ,$$
for all $2 \le i \le d$. Then $^{t}MQM = Q$, $\det\, M = 1$  
and the complex eigenvalues of $M$ are 
$$\sqrt{2} \pm 1\,\, ,\,\, e^{\pm 2\pi\sqrt{-1}\cdot\sqrt{2}} = \cos (2\pi\sqrt{2}) \pm \sqrt{-1}\sin (2\pi\sqrt{2})\,\, .$$
Note that $e^{\pm 2\pi\sqrt{-1}\cdot s}$ ($s \in \R$) is a root of unity if and only if $s$ is rational. Since $\sqrt{2}$ is an irrational real number, it follows that no eigenvalue of $M$ is root of unity. Thus $M$ satisfies all the requirements.
\end{proof}

Now $G \subset {\rm SO}^{+}(L)$. The argument  now closely follows \cite[p.~15]{BC13}.  

\begin{lemma} \label{lem:cyclotomic}
There are  finitely many cyclotomic polynomials of degree $\le 2d$.
\end{lemma}
\begin{proof}
This is because the number of complex numbers with $x^{(2d)!} = 1$ is at most $(2d)!$. 
\end{proof}

Let $P_{2d} \subset \Z [t]$ be the set of monic polynomials of degree $2d$. Then $P_{2d}$  is identified 
with  the affine variety $\A^{2d}$  defined over $\Z$.
The map  
$${\rm char} : {\rm SO}(L) \to P_{2d}\,\, ,\,\, h \mapsto \Phi_{h}(t) := {\rm det} (tI_{2d} -h)\,\, $$
is a morphism of affine varieties. 
Let $$u_1(t) = t-1, u_2(t) := t+1, \cdots, u_{N}(t)\,\, $$  be the cyclotomic polynomials in $\Z[t]$ of degree $\le 2d$, 
where $N$ is the cardinarity of the cyclotomic polynomials of degree $\le 2d$ (Lemma~\ref{lem:cyclotomic}). 
The subsets
$$P_i := \{p(t) \in P_{2d}(\C)\, | \  u_i(t) | p(t) \} $$
define
 proper closed algebraic subvarieties of $P_{2d}\otimes_{\Z} \Q$, thus so is their finite union 
$$Q_{2d} := \cup_{i=1}^{N} P_i \subset P_{2d}  \otimes_{\Z} \Q  \, .$$

Let $g \in G$. Its characteristic polynomial $ \Phi_g(t) \in \Z[t]$  is monic and of degree $2d$. By Proposition~\ref{salem}, $\Phi_g(t)$ is the product of cyclotomic polynomials and of at most one Salem polynomial counted with multiplicities. Thus, 
$\Phi_g(t)$ is a Salem polynomial of degree $2d$ if and only if $\Phi_g(t)$ is irreducible and is not a cyclotomic polynomial of degree $2d$, which is equivalent to saying that no $u_i(t)$  divides $\Phi_g(t)$   in $\Z [t]$. Since $\Phi_g(t)$ and $u_i(t)$ are monic polynomials in $\Z[t]$, it follows that no $u_i(t)$ divides
$\Phi_g(t)$  in ${\Z}[t]$ if and only if 
no $u_i(t)$ divides
$\Phi_g(t)$  in ${\C}[t]$.
The last condition is, by definition, equivalent to $ \Phi_g(t) \in P_{2d}(\C)\setminus  Q_{2d}$. 
The following lemma completes the proof:
\begin{lemma} \label{lem:complete}
There is an element $g \in G$ such that $\Phi_g(t)  \in P_{2d}(\C)\setminus Q_{2d}$. 
\end{lemma}
\begin{proof} By our assumption, 
we can apply Theorem \ref{thm:delaHarpe} to $G$, so  the Zariski closure of $G$ in ${\rm SO}(L_{\R})$ is ${\rm SO}(L_{\R})$.  On the other hand, ${\rm char}^{-1}( P_{2d} \otimes_{\Z} \Q \setminus Q_{2d})$  is Zariski open in ${\rm SO}(L)\otimes_{\Z}\Q$, and not empty by  Proposition~\ref{prop:evenness}. Thus it intersects $G \subset SO(L)(\Q)$ non-trivially. 
This finishes the proof. 
\end{proof}

This completes the proof of Theorem \ref{thm:BlancCantat}.

\end{proof}
The following theorem  is deduced from  \cite[Proof~of~Lem~3.6,~Claim~3.8]{Og09}.
\begin{thm} \label{thm:StableSublattice}
Let $L$ be a hyperbolic lattice of signature $(1, r+1)$ with $r \ge 0$ and let $e \in L$ be a primitive element such that $(e, e) = 0$. Let $G \subset {\rm SO}(L)(\Z)$ such that $G \simeq \Z^r$ as a group and  such that $g(e) = e$ for all $g \in G$. Then any $G$-stable $\R$-linear subspace $M$ of $L_{\R}$ is in the hypersurface $e^{\perp}$ in $L_{\R}$,  or $M = L_{\R}$. 
\end{thm}

\begin{proof} 
Note that $G \subset {\rm SO}^{+}(L).$ Indeed, for $g\in G$,  $g(e) =e$.  Any small enough open ball   in $L_{\R}$ in the classical topology, centered in $e$, meets $C$ in an open $\sU$, such that for any $x\in \sU$, the distance between $g(x)$ and $x$ is small enough so it forces $g(x) $ to lie in $ C$. 
\begin{lemma} \label{lem:basis}
There are an integral basis 
$$\langle e, w_1, \cdots , w_r \rangle$$
of $e^{\perp} \subset L$,  so necessarily $(w_j, w_j) < 0$ for all $1 \le j \le r$, 
a $\Q$-basis 
$$\langle e, w_1, \cdots , w_r, u \rangle$$
of $L_{\Q}$, with $(e, u) = 1$,  and a finite index subgroup 
$$H := \langle g_1, \cdots , g_r \rangle \simeq \Z^r$$ 
of $G$,  such that 
$$g_i = \left(\begin{array}{rrr}
1 & {\mathbf a}^t_i & c_i\\
{\mathbf 0} & I_r  & q_i{\mathbf e}_i\\
0 & {\mathbf 0}^{t} & 1
\end{array} \right)\,\, ,$$
under the $\Q$-basis above. Here $1, 0 \in \Q$ are the unit and the zero, $c_i$ and $q_i \not= 0$ are in $\Q$, ${\mathbf e}_i$ is the $i$-th unit vector of $\Q^r$, $I_r$ is the $r \times r$ identity matrix, ${\mathbf 0} \in \Q^r$ is the zero vector, ${\mathbf a}^t_i$ 
is the transpose of a column vector ${\mathbf a}_i \in \Q^r$, and simiarly for ${\mathbf 0}^{t}$. 
\end{lemma}
\begin{proof} This is observed in \cite[Proof of Lemma 3.6, Claim 3.8]{Og09}.  
The essential part is that $e^{\perp}/\Z e$ is a negative definite lattice with respect to the $\Z$-valued bilinear form induced by $(*, **)$ and $G$ acts on this negative definite lattice $e^{\perp}/\Z e$. Note then that 
$$H := {\rm Ker} \big(G \to {\rm O}(e^{\perp}/\Z e)(\Z) \big)$$
is a finite index subgroup of $G$ as ${\rm O}(e^{\perp}/\Z e)(\Z)$ 
is a finite group. 
Since $G \simeq \Z^r$, we have then $H \simeq \Z^r$ by the fundamental theorem of finitely generated abelian groups.  Next choose an $\Z$-basis 
$$\langle e, w_1, \cdots , w_r \rangle$$
of $e^{\perp}$. As $e^{\perp}$ is of signature $(0, 1, r-1)$ (degenerate lattice), it follows that $(w_i, w_i) < 0$ and $(e.w_i) = 0$, in addition to $(e^2) = 0$. Choose then $u \in L_{\Q}$ such that $(e.u) = 1$. Such a vector $u$ exists as $L$ is hyperbolic. Then, 
$$\langle e, w_1, \cdots , w_r, u \rangle$$
form a $\Q$-basis of $L_{\Q}$ and the matrix representation of $H$ with respect to the $\Q$-basis above is of the form 
$$g = \left(\begin{array}{rrr}
1 & {\mathbf a}^t(g) & c(g)\\
{\mathbf 0} & I_r  & {\mathbf b}(g)\\
0 & {\mathbf 0}^{t} & 1
\end{array} \right)\,\, ,$$
for all $g \in H \subset G\subset {\rm SO}(L)(\Z)$. It is proved in \cite[Claim~3.8]{Og09} that 
the map
$$H \to \Q^r\,\, ;\,\, g \mapsto {\mathbf b}(g)$$
is an injective group homomorphism. This is easily checked by an explicit computation of matrices. Since $H \simeq \Z^r$, it follows that 
$$\Z^r \simeq H \simeq \langle {\mathbf b}(g) | g \in H \rangle \subset \Q^r\,\, .$$
Thus, again by the fundamental theorem of finitely generated abelian groups, we obtain   an $\Z$-basis of $H$ with required property.   
\end{proof}
\begin{lemma} \label{lem:non-zero}
There are integers $i, j$ such that $1 \le i, j \le r$ and
$${\mathbf a}^t_j\cdot {\mathbf e}_i \not= 0\,\, .$$
Here the right hand side is the product as matrices and we naturally identify $1 \times 1$ matrices with the entry.\end{lemma}
\begin{proof} Now assume to the contrary that ${\mathbf a}^t_j\cdot {\mathbf e}_i = 0$ for all $i, j$. Then ${\mathbf a}^t_j = 0$ for all $j$. Thus 
$$g_j(u) = c_je + q_jw_j + u\,\,$$
by the explicit matrix form. Then by induction one has 
$$g_j^k(u) = g_j((k-1)c_je + (k-1)w_j +u) = k(c_je + q_jw_j ) + u$$
for all positive integer $k$. Since $g_j^k$ preserves intersection form, it follows that
$$(u, u) = (g^k(u), g^k(u)) = (u, u) + 2k (q_j(w_j, u) + c_j)  + k^2q_j^2(w_j, w_j)$$
whence
$$(w_j, w_j)q_j^2k + 2(c_j +  q_j(w_j, u)) =0$$
 for all positive integers $k$, a contradiction to $\big((w_j, w_j) < 0, \ q_j\neq 0\big)$ in Lemma \ref{lem:basis}. 
\end{proof}
To conclude the Theorem \ref{thm:StableSublattice}, it suffices to confirm the following:
\begin{lemma} \label{lem:non-proper}
Let $M$ be a $G$-stable $\R$-linear subspace of $L_{\R}$. Assume that there is $v \in M_{\R}$ such that $v \not\in e^{\perp}$ in $L_{\R}$. Then 
$M = L_{\R}$.
\end{lemma}
\begin{proof}  By replacing $v$ by multiple by $\R^{\times}$, we may assume without loss of generality that
$$v = xe + \sum_{s=1}^{r} y_sw_s + u$$
where $x$ and $y_s$ are real numbers. Then by Lemma~\ref{lem:basis}
$$g_i(v) = v + ({\mathbf a}^t_i\cdot{\mathbf y} + c_i)e +  q_iw_i\,\, ,$$
where ${\mathbf y} \in \R^r$ is the vector whose $s$-th entry is $y_s$. 
But $g(v) \in M$ by the assumption and by the fact that $M$ is $\R$-linear, 
it follows that
$$v_{i} := ({\mathbf a}^t_i\cdot{\mathbf y} + c_i)e +  q_iw_i \in M\,\, .$$
By Lemma \ref{lem:basis}, it follows that
$$g_j(v_{i}) = v_i + q_i(^{t}{\mathbf a}_j\cdot{\mathbf e}_i)e\,\, .$$
Recall from Lemma \ref{lem:basis} that $q_i \not= 0$. Then, for the same reason above, it follows that
$$({\mathbf a}^t_j\cdot{\mathbf e}_i)e \in M$$
for all $1 \le i, j \le r$. Thus by Lemma \ref{lem:non-zero}, $e \in M$. 
Combining this with $v_{i} \in M$, it follows that $q_iw_i \in M$ for all $i$. 
Since $q_i \not= 0$, it follows that $w_i \in M$ for all $i$. Combining this 
with $v \in M$, it follows that $u \in M$. Since $e, w_i, u$ form a $\Q$-basis of $L_{\Q}$, thus a  $\R$-basis of $L_{\R}$, it follows that $L_{\R} \subset M$. This finishes the proof.
\end{proof} 

This completes the proof of Theorem \ref{thm:StableSublattice}.
\end{proof}

\section{Proof of Theorems \ref{IntroThm2} and \ref{IntroThm1}}\label{sec:main}
\noindent

First we prove Theorem \ref{IntroThm2}. Let $X$ be as in Theorem \ref{IntroThm2}. We recall that ${\rm NS}\, (X)$ 
is a hyperbolic lattice of signature $(1, \rho(X) -1)$. 

Recall that an elliptic fibration $\varphi : X \to  \P^1$ on the minimal surface $X$ is a projective, surjective morphism over a field $k$, which has a section $O: \P^1\to X$,  and such that the generic fiber  is a smooth curve of genus $1$ over $k(\P^1)$.  We denote by ${\rm MW}(\varphi)$ its Mordell-Weil group, viewed it as a subgroup  of the group of the birational automorphisms ${\rm Bir}(X)$. 

The following lemma is well known. 
\begin{lemma} \label{lem:faithful}
Let $\varphi : X \to  \P^1$ be an elliptic fibration. Then 
\begin{itemize}
\item[(1)]
 ${\rm MW}\, (\varphi)$ of $\varphi$ is a finitely generated abelian subgroup of ${\rm Aut}\, (X) \subset {\rm Bir}(X)$;
\item[(2)]
the action of ${\rm MW}\, (\varphi)$ on ${\rm NS}\, (X)$ is faithful. 
\end{itemize}
\end{lemma}
We denote  by   ${\rm MW}\, (\varphi)^*\subset {\rm O}^{+}({\rm NS}\, (X))$  the image of the induced representation on the N\'eron-Severi group   ${\rm NS}\, (X)$  of $X$. Thus   ${\rm MW}\, (\varphi )\xrightarrow{\cong} 
{\rm MW}\, (\varphi)^*$. 

\begin{proof} The group  ${\rm MW}\, (\varphi)$   is finitey generated by \cite{Sh90}, as $\varphi$, for topological reasons,  has at least one singular fiber. As $X$ is a smooth projective minimal surface, ${\rm Bir}\, (X) = {\rm Aut}\, (X)$. This implies (1).  One has  $| O |  = \{O\}$, where $|O|$ is the complete linear system containinig $O$. This is because $O \simeq \P^1$ and $(O, O) = -2 < 0$ as $X$ is a K3 surface. 

Let $f \in {\rm MW}\, (\varphi)$. If $f^* | {\rm NS}\, (X) = {\rm Id}$, then  in particular, 
the class of $O$ in ${\rm NS}(X)$ is invariant under $f^*$, thus $|O|$ is invariant under $f$, thus the section $O$ is invariant under $f$. 
 As  $f \in {\rm MW}(\varphi)$, this implies $f = {\rm Id}$ on $X$. This proved (2). 
\end{proof} 
In the sequel, the notations  are as in Section~\ref{ss:salem}.
We define in addition
$${\rm MW}\, (\varphi)^{*0} := {\rm MW}\, (\varphi)^* \cap {\rm SO}^{+}({\rm NS}\,(X)_{\R}).$$

\begin{lemma} \label{lem:infinite}
\begin{itemize}
\item[(i)]
The group ${\rm Aut}\, (X)^{*0}$ is infinite, thus ${\rm Aut}\, (X)$ is infinite as well.  
\item[(ii)] The abelian  groups  ${\rm MW}\, (\varphi)$
and ${\rm MW}\, (\varphi)^{*0}$ have the same rank.
\end{itemize}

\end{lemma}

\begin{proof}  As ${\rm Aut}\, (X)^{*0} \supset  {\rm MW}\, (\varphi)^{*0} $, (i) follows from (ii).
We prove ii).  By  Lemma~\ref{lem:faithful}    ${\rm MW}\, (\varphi) \xrightarrow{\cong} {\rm MW}\, (\varphi)^*$. 
On the other hand, ${\rm SO}^{+} ({\rm NS}\, (X)_{\R})$ 
is a finite index subgroup of ${\rm O}({\rm NS}\, (X))(\R)$.  Thus ${\rm MW}\, (\varphi)^{*0}$ 
is a finite index sugroup of ${\rm MW}\, (\varphi)^*$.

\end{proof} 
Recall that  $X$ is as in Theorem~\ref{IntroThm2}. We denote by $e_1\in  {\rm NS}\, (X)$ the class of a fiber of $\varphi_1$.

\begin{lemma} \label{lem:inequality}
$X$ admits $g \in {\rm Aut}\, (X)$ such that $g^*(e_1) \not= e_1$ in ${\rm NS}\, (X)$. 
\end{lemma}

\begin{proof} By the assumption, $X$ admits a different elliptic fibration 
$\varphi_2 : X \to \P^1$ of positive Mordell-Weil rank. Thus, 
by Lemma~\ref{lem:infinite},
there is $g \in {\rm MW}\, (\varphi_2)$ of infinite order. 

Let $f \in {\rm NS}\, (X)$ be the fiber class of $\varphi_2$, thus in particular, $ g^*(f)=f$. 
By the Hodge index theorem, $(f+e_1, f+e_1) > 0.$ Therefore $(e_1 + f)^{\perp}$ in ${\rm NS}\, (X)$ is a negative definite lattice. 

Assume that $g^*(e_1) = e_1$ in ${\rm NS}\, (X)$. Then $g^{*}(e_1 + f) = e_1 + f$. Thus, by
 Lemma~\ref{lem:faithful}, $g$  acts faithfully on  $(e_1 + f)^{\perp}$. As ${\rm O}((e_1 + f)^{\perp})(
\Z)$ is a finite group, $g$ is  of finite order, a contradiction. This proves the Lemma.
\end{proof} 

\begin{definition} \label{def:maximum}
Consider all the elliptic fibrations $\Phi_i : X \to \P^1$ ($i \in I$) on $X$ 
with maximum Mordell-Weil rank $r = \rho (X) - 2$.  Let $e_i \in {\rm NS}\, (X)$ 
be the class of fibers of $\Phi_i$. Set
$${\mathcal S} := \{e_i \in {\rm NS}\, (X) | \  i \in I \},$$
and denote by $\R\langle {\mathcal S}\rangle \subset {\rm NS}(X)_{\R}$ the real sub vectorspace spanned by $\mathcal{S}$.
Note that $(e_i, e_i) = 0$, $e_i$ are numerically effective and $e_i $ are primitive in ${\rm NS}\, (X)$ for all 
$e_i \in {\mathcal S}$.
\end{definition}

\begin{lemma} \label{lem:positiveentropy} One has 
$\R\langle {\mathcal S}\rangle \xrightarrow{\cong}  
 {\rm NS}\, (X)_{\R}.$

 \end{lemma}

\begin{proof} Recall that ${\mathcal S} \not= \emptyset$ by our assumption. Let $e_1 \in {\mathcal S}$ be the class of $\Phi_1$.
By  definition of ${\mathcal S}$,   $\R\langle {\mathcal S} \rangle$ is ${\rm Aut}\, ( X)$-stable.
Let $g \in {\rm Aut}\, (X)$  with $e_2:=g(e_1)\neq e_1$ (Lemma~\ref{lem:inequality}).  As ${\mathcal S}$ is stable under ${\rm Aut}\, (X)$, it follows that $e_2 \in {\mathcal S}$ as well. 

Assume to the contrary that $\R\langle {\mathcal S} \rangle \not= {\rm NS}\, (X)_{\R}$. Let $G_1\subset  {\rm MW}\, (\Phi_1)^{*0}$ be a free abelian group of rank $r = \rho(X) -2$. By Theorem~\ref{thm:StableSublattice} applied to $G_1$, one has  $\R\langle {\mathcal S} \rangle \subset e_1^{\perp}$ in ${\rm NS}\, (X)_{\R}$.
By the Hodge index theorem and the fact that $e_i$ are  primitive, one has  $(e_1, e_2) > 0$, a contradiction.

\end{proof}

\begin{lemma} \label{lem:invariant}
There is no  ${\rm Aut}\, (X)^{*0}$-stable $\R$-linear subspace of ${\rm NS}\, (X)_{\R}$ other than $\{0\}$ and ${\rm NS}\, (X)_{\R}$.
\end{lemma}

\begin{proof} Let $M \neq  {\rm NS}\, (X)_{\R}$  be an  ${\rm Aut}\,(X)^{*0}$-stable $\R$-linear subspace of ${\rm NS}\, (X)_{\R}$.  For $\Phi_i$ as in Definition~\ref{def:maximum}, let   $G_i\subset {\rm MW}\, (\Phi_i)^{*0}$ be a free abelian subgroup of maximal rank $r = \rho(X) -2$  (Lemma~\ref{lem:faithful}). 
By Theorem \ref{thm:StableSublattice} applied to $G_i$, it follows 
$$M \subset \cap_{i\in I} e_i^{\perp} \subset {\rm NS}(X)_{\R}.$$
 As  the vectors  $e_i$,  for $i \in I$,   generate  ${\rm NS}\, (X)_{\R}$ and the intersection form is non-degenerate on ${\rm NS}\, (X)$, it follows that
$$\cap_{i \in I} e_i^{\perp} = \{0\}\,\, .$$
 This proves the lemma.
\end{proof}
\begin{proof}[Proof of Theorem~\ref{IntroThm2}]
By Lemma~\ref{lem:infinite} and Lemma~\ref{lem:invariant}, we can apply Theorem~\ref{thm:BlancCantat}.\end{proof}
\begin{proof}[Proof of Theorem~\ref{IntroThm1}]

As mentioned in the introduction, Shioda \cite{Sh13} proved that any $X(p)$ with 
$p = 11$ or $p > 13$ admits an elliptic fibration $\varphi_1 : X \to \P^1$ with Mordell-Weil rank $20 = 22 - 2$. On the other hand,  over $\F_p$, $X(p) \simeq {\rm Km}\, (E \times_{\F_p} E)$ for a supersingular elliptic curve $E$  over $\F_p$, and the fibration $\varphi_2 : X \to \P^1$ induced by the first projection $E \times_{ \F_p} E \to E$ is an elliptic fibration with  Mordell-Weil rank  $4$ (thus positive)    over $\bar \F_p$, by the formula of Mordell-Weil rank \cite{Sh90}. In particular, these two fibrations are non-isomorphic. So, we apply Theorem \ref{IntroThm2} to conclude Theorem \ref{IntroThm1} for $p = 11$ or $p >13$. The cases $p=2$ and $p=3$ are proved by \cite{BC13} and \cite{EO14}. This completes the proof of Theorem~\ref{IntroThm1}.

\end{proof}


\begin{thebibliography}{BlEsKe14} 


\bibitem[Ar74]{Ar74} Artin, M.: {\em 
Supersingular $K3$ surfaces},  
Ann. Sci. \'Ecole Norm. Sup. {\bf 7} (1975) 543--567. 




\bibitem[BC13]{BC13} Blanc, L., Cantat, S.: {\em Dynamical degrees of birational transformations of projective surfaces}, {\tt http://arxiv.org/pdf/1307.0361.pdf}. 

\bibitem[BH04]{BH04} Benoist, Y., de la Harpe, P.: {\em Adh\'erence de Zariski des groupes de Coxeter}, Compositio Math. {\bf 140}, (2004) 1357--1366.


\bibitem[SGA4.5]{SGA4.5} Deligne, P.: {\em S\'eminaire de G\'eom\'etrie Alg\'ebrique $4 \frac{1}{2}$: Cohomologie \'Etale}, Lecture Notes in Mathematics {\bf 569} (1977), Springer Verlag.

\bibitem[Del81]{Del81} Deligne, P.: {\em Rel\`evement des surfaces $K3$ en caract\'eristique nulle}, Lecture Notes in Math., {\bf 868},  (1981) 58--79. 


\bibitem[DK02]{DK02} Dolgachev,  I., Kondo, S.:  {\em A supersingular K3 surface of characteristic $2$ and the Leech lattice}, Int. Math. Res. Not. {\bf 1} (2003), 1--23. 




\bibitem[ES13]{ES13}  Esnault, H., Srinivas, V.: {\em  Algebraic versus topological entropy for surfaces over  finite  fields}, 
Osaka J. Math. {\bf 50} (2013) no3, 827--846. 

\bibitem[EO14]{EO14} Esnault, H., Oguiso, K.: {\em  Non-liftability of automorphism groups of a K3 surface in positive characteristic}, {\tt http://arxiv.org/pdf/1406.2761v3.pdf}. 


\bibitem[GMc02]{GMc02} Gross, B., McMullen C.T.:  {\em Automorphisms of even unimodular lattices and unramified Salem numbers},  J. Algebra {\bf 257} (2) (2002), 265--290.



\bibitem[SGA6]{SGA6} Grothendieck, A.:  {\em  S\'eminaire de G\'eom\'etrie Alg\'ebrique 6: Th\'eorie des Intersections and Th\'eor\`eme de Riemann-Roch,}  Lecture Notes in Mathematics {\bf 225} (1971), Springer Verlag.

\bibitem[Jan14]{Jan14} Jang, J.: {\em Representations of the automorphism group of a supersingular K3 surface of Artin-invariant $1$ over odd characteristic}, J. of Chungcheong Math. Soc. {\bf 27} 2 (2014), 287--295.


\bibitem[KS12]{KS12} S. Kondo, I. Shimada, : {\em The automorphism group of a supersingular K3 surface with Artin invariant 1 in characteristic 3}, {\tt http://arxiv.org/pdf/1205.6520v2.pdf}.

\bibitem[LM11]{LM11} Lieblich, M., Maulik, D.:  {\em A note on the cone conjecture for K3 surfaces in positive characteristic}, {\tt http://arxiv.org/pdf/1102.3377v3.pdf}.


\bibitem[Mc02]{Mc02} McMullen, C. T.: {\em Dynamics on $K3$ surfaces: Salem numbers and Siegel disks}, J. Reine Angew. Math. {\bf 545}  (2002) 201--233.

\bibitem[Mc11]{Mc11} McMullen, C. T.: {\em Automorphisms of projective K3 surfaces with minimum entropy}, {\tt http://www.math.harvard.edu/\~{}ctm/papers/home/text/papers/pos/pos.pdf}.

\bibitem[Og09]{Og09} Oguiso, K.: {\em Mordell-Weil groups of a hyperk\"ahler manifold - a question of F. Campana}, special volume dedicated to Professor Heisuke Hironaka on his $77$-th birthday, Publ. RIMS, {\bf 44} (2009) 495--506.

\bibitem[Og10]{Og10} Oguiso, K.: {\em Salem polynomials and the bimeromorphic automorphism group of a hyper-K\"ahler manifold},  Selected papers on analysis and differential equations, Amer. Math. Soc. Transl. Ser. {\bf 230} (2010) 201--227.



\bibitem[Ogu79]{Ogu79} Ogus, A.: {\em Supersingular $K3$ crystals},  Journ\'ees de G\'eom\'etrie Alg\'ebrique de Rennes, Ast\'erisque {\bf 64}, (1979), 3--86.

\bibitem[Ogu83]{Ogu83} Ogus, A.: {\em A crystalline Torelli theorem for supersingular $K3$ surfaces}, Progr. Math. {\bf 36} (1983) 361--394. 

\bibitem[Sch12]{Sch12}  Sch\"utt, M.: {\em A note on the supersingular K3 surface of Artin invariant 1},  

Journal of Pure and Applied Algebra {\bf 216}  (2012), 1438-1441. 
\bibitem[Sh75]{Sh75} Shioda, T.: {\em Algebraic cycles on certain K3 surfaces in characteristic $p$}, Manifolds-Tokyo 1973 (Proc. Internat. Conf., Tokyo, 1973),  357--364. Univ. Tokyo Press, Tokyo, 1975. 

\bibitem[Sh90]{Sh90} Shioda, T.: {\em On the Mordell-Weil lattices}, Comm.  Math. Univ. St. Paul {\bf 39}  2 (1990), 211--240.

\bibitem[Sh13]{Sh13} Shioda, T.: {\em Elliptic fibrations of maximal rank on a supersingular K3 surface},  Izv. Ross. Akad. Nauk Ser. Mat. {\bf 77}  (2013), 139--148;  translation in  Izv. Math.  {\bf 77}  (2013) 571--580.


 








\end{thebibliography}
\end{document}